\documentclass[11pt]{amsart}

\usepackage{amscd}
\usepackage{amsfonts,amssymb,latexsym}
\usepackage[all]{xy}

\xyoption{arc}

\CompileMatrices

\newcommand{\udots}{\mathinner{\mskip1mu\raise1pt\vbox{\kern7pt\hbox{.}}
\mskip2mu\raise4pt\hbox{.}\mskip2mu\raise7pt\hbox{.}\mskip1mu}}

\newcommand{\SC}{{\mathcal{C}}}
\newcommand{\SD}{{\mathcal{D}}}
\newcommand{\SE}{{\mathcal{E}}}
\newcommand{\SO}{{\mathcal{O}}}

\newcommand{\vicM}{{\mathcal{M}}}

\newcommand{\PP}{\mathbb{P}}

\newcommand{\CC}{\mathbb{C}}

\newcommand{\VV}{\mathbb{V}}

\newcommand{\Gr}{\operatorname{Gr}}
\newcommand{\Spec}{\operatorname{Spec}}

\newcommand{\codim}{\operatorname{codim}}

\newcommand{\im}{\operatorname{im}}
\newcommand{\surj}{\twoheadrightarrow}

\newcommand{\too}{\longrightarrow}
\newcommand{\rk}{\operatorname{rk}}
\newcommand{\End}{\operatorname{End}}
\newcommand{\wt}{\widetilde}
\newcommand{\tr}{\operatorname{tr}}
\newcommand{\GL}{\operatorname{GL}}
\newcommand{\Sp}{\operatorname{Sp}}
\newcommand{\Gp}{\operatorname{Gp}}

\newtheorem{proposition}{Proposition}[section]
\newtheorem{theorem}[proposition]{Theorem}

\newtheorem{lemma}[proposition]{Lemma}

\newtheorem{remark}[proposition]{Remark}

\numberwithin{equation}{section}

\begin{document}

\title[Automorphisms of moduli spaces]{Automorphisms of moduli
spaces of symplectic bundles}

\author[I. Biswas]{Indranil Biswas}
 \address{School of Mathematics, Tata Institute of Fundamental
 Research, Homi Bhabha Road, Bombay 400005, India}
\email{indranil@math.tifr.res.in}

\author[T. L. G\'omez]{Tom\'as L. G\'omez}
\address{Instituto de Ciencias Matem\'aticas (CSIC-UAM-UC3M-UCM),
Serrano 113bis, 28006 Madrid, Spain; and
Facultad de Ciencias Matem\'aticas,
Universidad Complutense de Madrid, 28040 Madrid, Spain. }
\email{tomas.gomez@icmat.es}

\author[V. Mu\~{n}oz]{Vicente Mu\~{n}oz}
\address{Facultad de Ciencias Matem\'aticas,
 Universidad Complutense de Madrid, 28040 Madrid, Spain}
\email{vicente.munoz@mat.ucm.es}


\date{}

\begin{abstract}
Let $X$ be an irreducible smooth complex projective curve
of genus $g\geq 3$. Fix a line bundle $L$ on $X$.
Let $M_{\rm Sp}(L)$ be the moduli space
of symplectic bundles $(E,\varphi:E\otimes E\to L)$ on $X$, with
the symplectic form taking values in $L$. We show that the
automorphism group of $M_{\rm Sp}(L)$ is generated by automorphisms
of the form $E\longmapsto E\otimes M$,
where $M^2\cong \SO_X$, and automorphisms induced by
automorphisms of $X$.
\end{abstract}

\maketitle

\section{Introduction}

Let $X$ be a smooth complex projective curve of genus $g$,
with $g\, \geq \, 3$.
A set of generators of the automorphism group of the moduli space of
semistable vector bundles over $X$ of rank $r$ with fixed determinant $L$
was obtained by Kouvidakis and Pantev in \cite{KP}. More precisely,
they proved that the automorphism group
is generated by the automorphisms of $X$,
automorphisms of the form $E\longmapsto E\otimes M$,
where $M$ is a line bundle with $M^{\otimes r}\cong \SO_X$, and,
if $r$ divides $2\deg L$, automorphisms of
the form $E\longmapsto E^\vee \otimes N$,
where $N$ is a line bundle with $N^{\otimes r}\cong L^{\otimes 2}$.
In the same paper they prove a Torelli theorem for these
moduli spaces. The proofs of their results crucially use
the Hitchin map defined
on the moduli of Higgs bundles.
In \cite{HR}, Hwang and Ramanan gave a different proof of the above
results using Hecke curves, which are minimal rational curves constructed
using Hecke transformations.

Fix a holomorphic line bundle $L$ on $X$, and consider the
moduli space $M_{\rm Sp}(L)$
of stable symplectic bundles $(E,\varphi:E\otimes E\too L)$ of rank $2n$ and with values
in $L$. Take a line bundle $M$ on $X$ with $M^{\otimes 2}\cong \SO_X$. Fix an
isomorphism $\beta : M^{\otimes 2} \longrightarrow \SO_X$. Then we have an
automorphism of $M_{\rm Sp}(L)$ defined by
$(E, \varphi) \longmapsto (E\otimes M, \varphi\otimes \beta)$.

More generally, let $\sigma:X\too X$ be an automorphism, and let $M$
be a line bundle on $X$
such that $M^{\otimes 2} \cong L\otimes (\sigma^*L)^\vee$. Fix
an isomorphism $\beta$ as above.
Then $(E, \varphi) \longmapsto (M\otimes \sigma^*E, \beta \otimes\sigma^*\varphi)$
is an automorphism of $M_{\rm Sp}(L)$. We remark that, in both cases,
the automorphism does not depend on the choice of $\beta$.

In Theorem \ref{thm:autom} we show that these are all the
automorphisms of $M_{\rm Sp}(L)$. More precisely,
the automorphism group $\text{Aut}(M_{\rm Sp}(L))$ fits in a short
exact sequence of groups
$$
e\too J(X)_2 \too \text{Aut}(M_{\rm Sp}(L)) \too \text{Aut}(X)
\too e\, ,
$$
where $J(X)_2$ is the group of line bundles on $X$ of order two
(see Proposition \ref{prop1}). Since the above
mentioned automorphisms of $M_{\rm Sp}(L)$ extend to the moduli
space of semistable symplectic bundles, it follows that
$\text{Aut}(M_{\rm Sp}(L))$ coincides with the
automorphism group of the moduli
space of semistable symplectic bundles (see Lemma \ref{lem-s-a}).

We also prove a Torelli type theorem for this moduli space
(Theorem \ref{thm:torelli}). This was proved earlier in \cite{BH}
by a different method.

Let us comment on our method of proof. The computation in \cite{KP}
of the automorphism group of the moduli space of vector bundles
uses a delicate argument in which the fibers of the Hitchin map
are studied over singular curves with non-generic singularities. Such
argument is not easy to generalize to other groups (like the symplectic
group). The proof of \cite{HR} is simpler in spirit: it determines
geometrically the Hitchin discriminant (the locus of singular
spectral curves), and then uses the theory of minimal rational
curves to prove that the dual variety of this Hitchin discriminant
is a locus of Hecke transforms. Neither of the theory of minimal
rational curves nor the constructions of Hecke transforms can be
generalized to other groups to an extent that cover the arguments.

We were lead to take the proof of \cite{HR}
for the automorphism group of the moduli space of vector bundles
and simplify it by removing the use of the dual varieties. Actually,
we found that the
Hitchin discriminant is enough to recover the nilpotent cones, and
from this to get the automorphism group.
Therefore, the proof given in this paper is the extension to the case
of the moduli space of \emph{symplectic} vector bundles of a proof
for the moduli space of vector bundles which is not in the
literature, and which simplifies both \cite{KP} and \cite{HR}.

\section{Moduli space of symplectic bundles}

Let
$$
J\,=\,
\begin{pmatrix}
0_{n\times n} & I_{n\times n}\\
-I_{n\times n} & 0_{n\times n}
\end{pmatrix}
$$
be the standard symplectic form on $\CC^{2n}$.
Define the group
\begin{equation}\label{eq.-c}
\Gp(2n,\CC)=\,\big\{ A\in \GL(2n,\CC)\,:\, A^t J A= c J \;
\text{for some}\; c\in \CC^*\big \}\, .
\end{equation}
It is an extension of $\CC^*$ by the symplectic group
$\Sp(2n,\CC)$
$$
e\too\Sp(2n,\CC) \too\Gp(2n,\CC) \stackrel{q}{\too}\CC^*\too e\, ,
$$
where $q(A)=c$ \ for any $A$ and $c$ as in \eqref{eq.-c}. From the
definition of the homomorphism $q$ it follows immediately that
for all $A\, \in\, \Gp(2n,\CC)$,
\begin{equation}\label{det.-id.}
\det A\, =\, q(A)^n\, .
\end{equation}

Let $X$ be an irreducible smooth complex projective curve of
genus $g$, with $g\geq 3$.
A symplectic bundle on $X$ of rank $2n$ with values in a
holomorphic line bundle $L$ is a pair
$(E,\varphi)$, where
$E$ is a holomorphic vector bundle of rank $2n$ and
 $$
 \varphi\,:\,E\bigwedge E\,\too\, L
 $$
is a homomorphism of coherent sheaves which is fiberwise nondegenerate.
The line bundle $\det(E)$
is canonically a direct summand of $(E\bigwedge E)^{\otimes n}$,
and the composition
$$
\det(E)\, \hookrightarrow\, (E\bigwedge E)^{\otimes n}\,
\stackrel{\varphi^{\otimes n}}{\longrightarrow}\, L^{\otimes n}
$$
is an isomorphism.
Giving a symplectic bundle is equivalent to giving a principal
$\Gp(2n,\CC)$-bundle.

Let $(E,\varphi)$ be a symplectic bundle. A holomorphic
subbundle $F$ of $E$ is called \textit{isotropic} if
$\varphi(F\bigwedge F)\,=\, 0$.

A symplectic bundle $(E,\varphi)$ is called \textit{stable}
(respectively, \textit{semistable}) if, for
all isotropic proper subbundles $E'\subset E$ of positive rank,
$$
\frac{\deg E'}{\rk E'} < \frac{\deg E}{\rk E} \quad
\text{(respectively, $\frac{\deg E'}{\rk E'}\leq \frac{\deg E}{\rk E}$)}
$$
See \cite{BG} for more on symplectic bundles.

We denote by $M_{\Sp}(L)$ the moduli space of stable symplectic bundles
with values in a fixed line bundle $L$.

\begin{lemma} \label{lem:asdfg1}
Assume that $\deg L \,\leq\, 2(g-1)$. Then
$H^0(E)=0$ for
a general stable bundle $(E,\varphi)\,\in\, M_{\Sp}(L)$.
\end{lemma}

\begin{proof}
Using Riemann--Roch, $\dim M_{\Sp}(L)\,=\, n(2n+1)(g-1)$. By
semicontinuity,
 $$
 \big\{(E,\varphi)\,\in\, M_{\Sp}(L)\, \mid\, H^0(E)\, \not=\, 0\big\}
 \, \subset\, M_{\Sp}(L)
 $$
is a Zariski closed subset. The lemma will be proved by showing that
the codimension of this subset is positive.

Take a pair
$((E,\varphi)\, ,s)$ such that $(E,\varphi)\,\in\,M_{\Sp}(L)$ and
$s\,\in\, H^0(E)\setminus\{0\}$. It defines a short exact sequence
 \begin{equation}\label{eqn:asdfg-ext}
0\too M \,=\, {\mathcal O}_X(D)\stackrel{s}{\too} E \too Q \too 0\, ,
\end{equation}
where $D$ is the effective divisor defined by $s$.
Let $K$ be the kernel of the composition
$$
E \stackrel{\varphi}{\too} E^\vee\otimes L \too M^\vee\otimes L\, .
$$
Define $Q\,:=\, E/M$. Since $\varphi(M\otimes K) \, =\, 0$, it
follows that $\varphi$ defines a pairing
$$
Q\otimes K\, \longrightarrow\, L\, .
$$
This pairing is perfect because $\varphi$ is pointwise nondegenerate.
In particular, $Q\cong K^\vee \otimes L$. We have a diagram
\begin{equation}\label{e1}
 \begin{array}{ccccccccc}
&&&& 0 && 0 &&\\
&&&&\Big\downarrow && \Big\downarrow &&\\
0 &\too & M & \too & K &\too & F &\too & 0\\
 && || && \Big\downarrow && \Big\downarrow && \\
0 &\too & M &\too & E &\too & Q &\too & 0\\
 &&&& \Big\downarrow && \Big\downarrow \\
 &&&& M^\vee \otimes L &=& M^\vee \otimes L\\
&&&&\Big\downarrow && \Big\downarrow &&\\
&&&& 0 && 0 &&\, ,
 \end{array}
\end{equation}
and there is a symplectic form $\omega_F\, :\, F\otimes F\,\too
\,L$ induced by $\varphi$ (recall that $\varphi(M\otimes K) \, =\, 0$).

Note that under the homomorphism
 $$
 {\rm Ext}^1(Q,M)\too {\rm Ext}^1(F,M)={\rm Ext}^1(M^\vee \otimes L,
F)\, ,
 $$
the class $\xi_1\,\in\, {\rm Ext}^1(Q,M)$ for the bottom exact sequence
in \eqref{e1} maps to the class $\xi_2\,\in\, {\rm Ext}^1(M^\vee \otimes L,
F)$ for the vertical exact sequence in the right of \eqref{e1}.
For a general point $(E',\varphi')\,\in\,M_{\Sp}(L)$, the underlying
vector
bundle $E'$ is stable. If $E$ is a stable vector bundle, then
${\rm Hom} (Q,M)\,=\,0$, because any nonzero homomorphism from
$Q$ to $M$ produces a nilpotent endomorphism of $E$.

Let $\deg M\,=\,\ell$ and $\deg E\,=\,n\cdot \deg L \,=\, d$. If
${\rm Hom} (Q,M)\,=\,0$, then
\begin{equation}\label{eqn:asdfg-dimext1}
\dim {\rm Ext}^1(Q,M)= -\deg(Q^\vee \otimes M) + (2n-1)(g-1)= d -2 n
\ell +
(2n-1)(g-1)\, .
\end{equation}

Now let us see that the symplectic form $\omega_F:F\otimes
F\longrightarrow L$
determines the symplectic form on $E$.
First, $\omega_F$ extends
uniquely to a homomorphism $F\otimes K\longrightarrow L$, which extends
naturally to a homomorphism
$Q\otimes K\longrightarrow L$; both these extensions are consequences
of the fact that $\varphi(M\otimes K)\,=\, 0$. Any two extensions
of the pairing $F\otimes K\longrightarrow L$ to
a pairing $Q\otimes K\longrightarrow L$ differ by a section
contained in
${\rm Hom}((Q/F)\otimes K,L)={\rm Hom}( (M^\vee\otimes L)\otimes K
,L)={\rm Hom}(K,M)$.

We will show that
\begin{equation}\label{c0}
{\rm Hom}(K,M)\,=\,0\, .
\end{equation}
First, if a
homomorphism $K\, \longrightarrow\, M$ composed with
$M\, \longrightarrow\, K$ is non-zero, then it
produces a splitting of
the short exact sequence
$$
0 \longrightarrow M\longrightarrow K\longrightarrow F\longrightarrow
0\, .
$$
So the extension $F \longrightarrow Q \longrightarrow M^\vee\otimes L$ is
split.
Therefore there are maps $Q\longrightarrow F$ and $F\longrightarrow K$, which composed with
$K\longrightarrow E$ splits the diagram $Q\longrightarrow E$, but
this is not possible since $E$ is stable.
So the homomorphism $K\, \longrightarrow\, M$ composed with $M\too K$
is the zero homomorphism. But then the homomorphism
$K\, \longrightarrow\, M$
descends to a homomorphism $F\too M$. Let $S_1$ (respectively, $S_2$) be the kernel of $K\too M$
(respectively, of $F\too M$). Then there is an exact sequence $$0\too M\too S_1\too S_2\too 0\, .$$ So it follows
that $\deg F \leq \deg S_1$.
As $S_1\subset K\subset E$, and
$E$ is a stable bundle, then $\mu(S_1)<\mu(E)$. So $$\frac{\deg L}{2} = \mu(F) \leq \mu(S_1) <\mu(E)=
\frac{\deg L}{2}\, ,$$ which is a contradiction. Therefore, \eqref{c0}
is proved.

Now the homomorphism $Q\otimes K\too L$ extends uniquely to a map $E\otimes K\too L$. This again extends to the map
$\omega_E:E\otimes E\too L$, up to an indeterminacy contained in ${\rm Hom}(E\otimes (E/K),L)=
{\rm Hom}(E\otimes (M^\vee\otimes L),L)$, which actually lives in the subspace ${\rm Hom}((E/M)\otimes (M^\vee\otimes L),L)=
{\rm Hom}(Q,M)=0$.

Then the dimension of the family of bundles parametrizing (\ref{eqn:asdfg-ext}) is
 \begin{eqnarray*}
 &&(2n-1) (n-1) (g-1) +d -2 n \ell + (2n-1)(g-1) -1 \\
&&\leq (2n+1)n(g-1) + d-2n(g-1) -1< (2n+1)n(g-1),
 \end{eqnarray*}
for $d\leq 2n(g-1)$. This completes the proof of the lemma.
\end{proof}

For a symplectic bundle $(E,\varphi)$, let
$$
\End_{\Sp}(E) \, :=\, \text{Sym}^2(E)\otimes L^\vee \, \subset \,
\End(E) \,=\, E\otimes E\otimes L^\vee
$$
be the set consisting of symmetric symplectic endomorphisms of $E$.
For any divisor $D$ on $X$, define
$$
\End_{\Sp} E (D)\, :=\, \End_{\Sp}(E)\otimes_{{\mathcal O}_X}
{\mathcal O}_X(D)\, .
$$

\begin{lemma}\label{lem:asdfg2}
Let $D$ be an effective divisor of degree $\ell$, with
$g\,\geq\, \max\{2\ell,\ell+2\}$. Then
$H^0(\End_{\Sp} E (D))=0$ for a general stable symplectic bundle
$(E,\varphi)\in M_{\Sp}(L)$.
\end{lemma}

\begin{proof}
By tensoring with a suitable line bundle, we may assume that $L$ has
degree $\epsilon\, \in\, \{0,1\}$.
This makes the slope of any symplectic bundle to be $\frac{\epsilon}2<g-1$.

Moreover,
we may assume that $L$ is generic in the sense that
$$H^0(L(D))=0 ~\,~\, \text{~and~}~\,~\,
H^0(L^*(D))=0\, , $$
for any $D$ effective divisor of degree $\ell \leq g-1-\epsilon$.
Let $X^{(\ell)}\,=\, \text{Sym}^\ell(X)$ be the
set of effective divisors of degree $\ell$.

For $n=1$, consider any stable vector bundle $F$ of rank two with
determinant $L$.
It has a symplectic structure:
 $$
\omega: F\otimes F \too \wedge^2 F=L\, .
 $$
The symplectic bundle $(F,\omega)$ is automatically stable. However, we
are going to construct an
specific bundle $F_0$ for later use. Consider an extension
 \begin{equation}\label{eqn:asdfg-f0}
 0\too {\mathcal O}_X \too F_0 \too L \too 0\, .
 \end{equation}
These extensions are parametrized by elements in $H^1(L^*)$, and
$\dim H^1(L^*)= g-1+\epsilon$.
Consider an effective divisor $D\in X^{(\ell)}$.
Then the exact sequence
 $$0\too L^* \too L^*(D)\too L^*(D)\vert_D \too 0$$
gives an exact sequence
 $$0\too L^*(D)\vert_D \too H^1(L^*)\too H^1(L^*(D))\too 0\, ,$$
so we get a subspace
$V_D\,:=\, L^*(D)\vert_D \subset\,
H^1(L^*)$ of dimension $\ell$. Moving $D$
over $X^{(\ell)}$, we see that if $g-1+\epsilon> 2\ell$, then there is
an
extension (\ref{eqn:asdfg-f0})
whose class $\xi \in H^1(L^*)$ goes to a non-zero element under
the homomorphism $H^1(L^*)\too H^1(L^*(D))$ for any $D\in X^{(\ell)}$.
Now the connecting homomorphism $H^0({\mathcal O}_X(D)) \too
H^1(L^*(D))$ for the dual sequence
 \begin{equation}\label{eqn:asdfg-f}
 0\too L^*(D) \too F_0^*(D) \too {\mathcal O}_X(D) \too 0
 \end{equation}
of (\ref{eqn:asdfg-f0}),
which is multiplication by $\xi$,
is injective; indeed, if a section $s\in H^0({\mathcal O}_X(D))$,
defining a divisor $D'$, maps to zero, then the extension class $\xi$ goes to
zero under the homomorphism $H^1(L^*)\too H^1(L^*(D'))$, but this is not
the case by construction. This implies that
 $$
 H^0(L^*(D))= H^0(L^*\otimes F_0(D))\, ,
 $$
which is zero by assumption.
Also the exact sequence $$0\too H^0({\mathcal O}_X(D))\too
H^0(F_0(D))\too H^0(L(D))=0$$ implies that
$H^0(F_0(D))= H^0({\mathcal O}_X(D))$. And finally, the exact sequence
 $$
{\rm Hom} (L, F_0(D))=0 \too {\rm Hom} (F_0,F_0(D))\too {\rm
Hom}({\mathcal O}_X, F_0(D))= H^0({\mathcal O}_X(D))\,
 $$
gives that $H^0(\End F_0(D))=H^0({\mathcal O}_X(D))$. Hence
$H^0(\End_{0} F_0(D))=0$, where $\End_0$ denotes the space of trace-free
endomorphisms. Note that
$\End_{\Sp}F_0=\End_0 F_0$, so $H^0(\End_{\Sp} F_0(D))=0$.

Now for $n>1$, consider a general symplectic bundle $(F_1,\omega_1)$ of rank $2n-2$.
By induction hypothesis, $H^0(\End_{\Sp} F_1(D))=0$, for any effective divisor $D$
of degree $\ell$.
Consider the symplectic bundle $E=F_0 \oplus F_1$. This is a symplectic
semistable bundle of rank $2n$. Let us see that
\begin{equation}\label{eqn:asdfg-end}
 H^0(\End_{\Sp}E(D))=0\, .
\end{equation}
This would imply that also for a general stable bundle $\widetilde{E}$, we
have that
 $$H^0(\End_{\Sp}\widetilde{E}(D))=0$$
for all $D\in X^{(\ell)}$ (note that
$X^{(\ell)}$ is a complete variety).

The vector space $H^0(\End_{\Sp}E(D))$ has four
components:
 \begin{itemize}
 \item $H^0(\End_{\Sp}F_0(D))=0$, by construction.
 \item $H^0(\End_{\Sp}F_1(D))=0$, by induction hypothesis.
 \item $H^0({\rm Hom}_{\Sp}(F_0,F_1(D)))=0$. A homomorphism $\varphi:F_0\too
F_1(D)$ can be restricted to ${\mathcal O}\subset F_0$,
so it defines a section in $H^0(F_1(D))$. By Lemma \ref{lem:asdfg1},
this is zero (as
$\mu(F_1(D))= \frac{\epsilon}2 +\ell <g-1$). So $\varphi$
defines a section of the quotient $L \too F_1(D)$, i.e., a section of
$H^0(L^* \otimes F_1(D))$, which is also zero
($L$ is fixed, so can take both $F_1$ and $F_1 \otimes L^*$ to be
simultaneously generic).
\item $H^0({\rm Hom}_{\Sp}(F_1,F_0(D)))=0$.
A homomorphism $\varphi:F_1\too F_0(D)$ gives a homomorphism
$F_0^\vee=F_0 \otimes L^{-1} \too F_1^\vee(D)=F_1 \otimes L^{-1}(D)$,
i.e., a symplectic map
$F_0 \too F_1(D)$, which is zero as above.
\end{itemize}
This completes the proof of the lemma.
\end{proof}

\section{Hitchin discriminant}

Let us recall the definition of the Hitchin map (see \cite[Section
5.10]{Hi}). The holomorphic cotangent bundle of $X$ will be
denoted by $K_X$. A symplectic Higgs bundle is a triple
$(E,\omega, \theta)$, where
$(E,\omega)$ is a symplectic bundle and $\theta:E\too E\otimes K_X$ is a
symmetric map
with respect to $\omega$:
 $$
 \omega(u,\theta(v)) = - \omega(\theta(u),v)
 $$
for $u, v\in E_x$, $x\in X$.

Let $\vicM_{\Sp}(L)$ be the moduli space
of semistable symplectic Higgs bundles of rank $2n$.

As before, $M_{\Sp}(L)$ is the
moduli space of stable symplectic bundles.
The cotangent bundle $T^* M_{\Sp}(L)\subset \vicM_{\Sp}(L)$ is
an open subset.
Consider the affine space:
 $$
 W= H^0(K_X^2) \oplus \ldots \oplus H^0(K_X^{2n})\, ,
 $$
and the \textit{Hitchin map on $T^* M_{\Sp}(L)$}
 $$
 h:T^* M_{\Sp}(L) \too W \, ,
 $$
defined by
$h(\theta)=(s_2(\theta),\ldots, s_{2n}(\theta))$, where
$s_i(\theta)\,=\,{\rm tr}(\wedge^i \theta)$, and
 $$
 \theta\in T^*_E M_{\Sp}(L)= H^0(\End_{\Sp}(E)\otimes K_X)\, .
 $$
This extends to the \textit{Hitchin map} on the moduli
space $\vicM_{\Sp}(L)$ of semistable symplectic Higgs bundles,
 $$
 H: \vicM_{\Sp}(L) \too W \, .
 $$

For an element $s=(s_2,\ldots, s_{2n}) \in W$, the \textit{spectral curve} $X_s$ associated to $s$ is
the curve in the total space $\VV(K_X)$ of $K_X$ defined by the equation
\begin{equation}\label{eqn:spectral-curve}
y^{2n} + s_2(x) y^{2n-2} + \ldots + s_{2n-2}(x) y^2 + s_{2n}(x)=0
\end{equation}
($x$ is a coordinate for $X$, and $y$ is the corresponding tautological
coordinate $dx$ along the fibers of the projection $\VV(K_X)\too X$).

Consider the compactification
$$S:=\PP(\SO_X\oplus K_X) \, \subset\, \VV(K_X)\, .$$
Let $p:S\too X$ be the projection.
Giving a Higgs bundle $(E,\theta:E\to E\otimes K_X)$ is equivalent
to giving a coherent sheaf $A$ of rank one supported on some spectral
curve $S\, \subset\,
\VV(K_X)$.
Indeed, $E=p_* A$, and the Higgs field $\theta$ corresponds to
the homomorphism $A\, \longrightarrow\, A\otimes p^*K_X$
defined by multiplication with the tautological section of
$p^*K_X$ over $\VV(K_X)$ (recall that $S\, \subset\,\VV(K_X)$).
The support of $A$ is given by
the equation \eqref{eqn:spectral-curve}. For more details,
see \cite{Hi}, \cite{BNR} and \cite{Si}.

The symplectic bundle structure $\omega:E\otimes E\too L$ corresponds
to an isomorphism
$$
\sigma^* A \stackrel{\cong}\too Ext^1(A,K_S\otimes p^*K_X) \otimes p^* L
$$
where $\sigma:S\too S$ is the involution $y\longmapsto -y$
(note that the spectral curve is invariant
under this involution because all the exponents of $y$ in
\eqref{eqn:spectral-curve} are even integers).
Indeed, applying $p_*$ to this isomorphism we obtain the
symplectic structure:
$$
p_*\sigma^* A = E \too p_*(Ext^1(A,K_S\otimes p^*K_X^{-1})) \otimes
L=E^\vee \otimes L \; .
$$
The second equality is proved in two steps. There is a spectral sequence
$$R^ip_* Ext^j(\cdot,\cdot) \Rightarrow Ext_p^{i+j}(\cdot,\cdot)\, .$$
Since $A$ has support of dimension 1, we obtain
$$
p_*(Ext^1(A,K_S\otimes p^*K_X^{-1}))= Ext_p^1(A,K_S\otimes
p^*K_X^{-1})\, ,
$$
and the relative Serre duality for the projective morphism $p$
gives
$$
Ext_p^1(A,K_S\otimes p^*K_X^{-1})= p_*(A)^\vee=E^\vee \; .
$$
We can think of $A$ as a sheaf on the spectral curve $X_s$.
If this is integral, then $A$ is torsionfree as a sheaf
on $X_s$, and then
$$
Ext^1(A,K_S\otimes p^*K_X)\,=\, A^\vee \otimes K^{}_{X_s}\otimes \pi^*
K_X^{-1}\, ,
$$
where $\pi:X_s\too X$ is the
projection. For an arbitrary coherent sheaf $A$ on $S$ supported
on $X_s$, define $$A^\vee\, :=\, Ext^1(A,K_S)\otimes
Ext^1(\SO_{X_s},K_S)^\vee\, .$$
If $A$ is locally free on $X_s$, then $A^\vee$ is the
usual dual line bundle on $X_s$.

Fix once and for all a square root line bundle $R=(K^{}_{X_s}\otimes
\pi^*K_X^{-1}\otimes \pi^*L)^{1/2}$.
If we denote $U=A\otimes R$, then
$\sigma^*U \cong U^\vee$. In other words, $U$ is an element
of the Prym subvariety of the compactified Jacobian
$\overline{J}(X_s)$
$$
{\rm Prym}(X_s,\sigma)=\big\{ U \in \overline{J}(X_s): \sigma^* U\cong
U^\vee \big\}\, ,
$$
and, conversely, an element of this Prym produces a symplectic
Higgs bundle whose spectral curve is $X_s$.
Therefore, the fiber of $H$ over $s\in W$ is
isomorphic to ${\rm Prym}(X_s,\sigma)$, and the
isomorphism depends only on the choice of square root $R$.
The dimension of this Prym variety is
 $$
 g(X_s)-g(X_s/\sigma)=n(2n+1)(g(X)-1)= \dim \Sp(2n) (g-1)\, .
 $$

Let $Y$ be an integral curve whose only singularity is one
simple node at a point $y$. Let $$\pi_Y:\wt Y\too Y$$ be the
normalization,
and let $x$ and $z$ be the pre-images of $y$ in $\wt Y$. The
compactified
Jacobian $\overline{J}(Y)$, which parametrizes torsionfree sheaves of
rank 1 and degree 0 on $Y$, is birational to a $\PP^1$-fibration
$P$ over $J(\wt Y)$, whose fiber over any $L\,\in\, J(\wt Y)$
is $\PP^1(L_x\oplus L_z)$. The morphism $P\too \overline{J}(Y)$
is constructed as follows. A point of $P$ corresponds to a
line bundle $L$ on $\wt Y$ and a one dimensional quotient
$q:L_x\oplus L_z \surj \CC$ (up to scalar multiple).
This is sent to the torsionfree sheaf
$L'$ on $Y$ defined as
$$
0 \too L' \too (\pi_Y)_*L \stackrel{q}\too \CC_y \too 0\, .
$$
For the proof, see \cite[Theorem 4]{Bh}.

Assume that $Y$ has an involution $\sigma$. It lifts to an involution
$\wt\sigma$ of $\wt Y$. This induces an involution in $P$. Indeed,
if $(L,q)$ is a point, and $q:L_x\oplus L_z\surj \CC$
is represented by $[a:b]$, then this point $(L,q)$ is
sent to $(\wt\sigma^* L^\vee,q^\vee:=[b:a])$.
Note that the definition of $q^\vee$ makes sense:
if $[a:b]\in \PP(L_x\oplus L_z)$,
then
$$
[b:a]\in \PP(L_z\oplus L_x)=
\PP(L_x^\vee\otimes L_z^\vee\otimes (L_z\oplus L_x))=
\PP(L_x^\vee\oplus L_z^\vee)\, .
$$
The involution on $P$ induces an involution in $\overline{J}(Y)$,
which restricts to $A\longmapsto \sigma^* A^\vee$ on the open
subset $J(Y)\subset\overline{J}(Y)$ corresponding to
line bundles. The fixed point variety of this involution
is the Prym variety
$$
{\rm Prym}(Y,\sigma)\subset \overline{J}(Y) \; .
$$
It is a uniruled variety, because it has a surjective
morphism from the $\PP^1$ fibration $P\vert_{\rm Prym}$
defined by the pullback
$$
\xymatrix{
{P\vert_{\rm Prym}} \ar[r] \ar[d] & {P}\ar[d]\\
{\rm Prym}(\wt Y,\wt\sigma)\ar[r] & J(\wt Y)
}
$$

Analogously, if $Y$ is an integral curve with two simple nodes,
and $\wt Y$ is its normalization, then $\overline{J}(Y)$
is birational to a $\PP^1\times \PP^1$-bundle $P$ on $J(\wt Y)$. If the
nodes are called $y_1$ and $y_2$, then the two $\PP^1$-factors
in the Cartesian product
correspond to one dimensional quotients
$q_1:L_{x_1}\oplus L_{z_1}\surj \CC$ and
$q_2:L_{x_2}\oplus L_{z_2}\surj \CC$, where $\{x_i\, ,z_i\}
\,\subset\, \wt Y$ is the inverse image of $y_i$.

Let $\sigma$ be an involution of $Y$ interchanging the two nodes.
It lifts to an involution of $\wt Y$, and also to an involution
of $P$, sending $(L,q_1,q_2)$ to $(\wt\sigma^* L^\vee
,q_2^\vee,q_1^\vee)$, and this induces an
involution of $\overline{J}(Y)$. A fixed point $(L,q_1,q_2)$ in $P$
for this involution has $L\cong \wt\sigma^*L ^\vee$
and $q_2^{}=q_1^\vee$, hence it
is a $\PP^1$-fibration on ${\rm Prym}(\wt Y)$, and the image
of this map is the fixed point locus on
$\overline{J}(Y)$, which is denoted by ${\rm Prym}(Y,\sigma)$.
We again obtain that this Prym is a uniruled variety.

Consider
 $$
 {\mathcal D}\subset W
 $$
the divisor consisting of characteristic polynomials with singular
spectral curves. This has two components
 $$
 {\mathcal D}={\mathcal D}_1 \cup {\mathcal D}_2\, ,
 $$
where ${\mathcal D}_1$ consists of those curves for which $s_{2n}$ has
a double root (then (\ref{eqn:spectral-curve})
has a node at the horizontal axis), and ${\mathcal D}_2$ consists of curves
with two symmetrical nodes (i.e., $y^{n} + s_2(x) y^{n-1} + \ldots +
s_{2n-2}(x) y + s_{2n}(x)=0$ has
a node). Let ${\mathcal D}^o_i\, \subset\, {\mathcal D}_i$,
$i=1,2$, be the locus of all those curves that
do not contain extra singularities. Finally let
${\mathcal D}^*={\mathcal D} -({\mathcal D}_1^o\bigcup {\mathcal
D}^o_2)$.

\begin{proposition} \label{prop:W-and-D}
As before, $h:T^* M_{\Sp}(L) \too W$ is the Hitchin map. The following
statements hold:
\begin{enumerate}
 \item For $w\in W-{\mathcal D}$, the fiber $h^{-1}(w)$ is an open subset of an abelian variety (actually a Prym variety).
 \item For $w\in {\mathcal D}^o_1$, the fiber $h^{-1}(w)$ is an
open subset of the uniruled variety ${\rm Prym}(X_w,\sigma)$.
 \item For $w\in {\mathcal D}^o_2$, the fiber $h^{-1}(w)$ is an
open subset of the uniruled variety ${\rm Prym}(X_w,\sigma)$.
\end{enumerate}
The complement of the open subsets in each of the cases is of
codimension at least $2$ (at least for generic $w$ in the corresponding set).
\end{proposition}

\begin{proof}
The map $H:\vicM_{\Sp}(L) \too W$ is proper.
By \cite{Hi}, $H^{-1}(w)$ is an abelian variety for $w\in W - {\mathcal
D}$. The complement
 $$
 \vicM_{\Sp}(L)- T^*{M}_{\Sp} (L)
 $$
is of codimension $\geq 3$ (the assumption that $g\geq 3$ is used
here).
In \cite[Theorem II.6 (iii)]{Fa} it is proved that the complement
has codimension $\geq 2$ under a weaker assumption, but if
we assume $g\geq 3$, then the same proof gives that the codimension
is $\geq 3$.

Therefore, $(\vicM_{\Sp}(L)- T^*{M}_{\Sp}(L))\cap {\mathcal D}_i$
is
of codimension at least $2$
in ${\mathcal D}_i$,
so for generic $w\in {\mathcal D}^o_i$,
$$H^{-1}(w)-h^{-1}(w)\,\subset\, H^{-1}(w)$$ is of codimension at least
$2$.

The computations of $H^{-1}(w)$ for $w\in {\mathcal D}^o_i$ were done
in the arguments above.
\end{proof}

\begin{proposition} \label{prop:irreducibility}
The hypersurfaces $h^{-1}({\mathcal D}_i)$ are irreducible.
\end{proposition}

\begin{proof}
We need to see that $h^{-1}({\mathcal D}^*)$ is of codimension at least
two in $T^*M_{\Sp}(L)$.
This follows easily from Theorem II.5 of \cite{Fa}, which says that the
fibers of the Hitchin map $H:\vicM_{\Sp}(L)\too W$ are Lagrangian
(hence
of half-dimension). So the fibers of $H$ are equidimensional, and in particular
the codimension of $h^{-1}({\mathcal D}^*)$ is that of ${\mathcal D}^*\subset
W$, which is at least two.
\end{proof}

The inverse image $h^{-1}({\mathcal D})$ is called the
\textit{Hitchin discriminant}.

\begin{theorem} \label{thm:Hitchin-discriminant}
 The Hitchin discriminant $h^{-1}({\mathcal D})$ is the closure of the union of
 the (complete) rational curves in $T^* M_{\Sp}(L)$.
\end{theorem}

\begin{proof}
 Let $l\cong \PP^1\subset h^{-1}({\mathcal D})$. Then $h(l)\subset W$. As it is
a complete curve, it should be a point. So $l$ is included in a fiber. By Proposition
\ref{prop:W-and-D}, it cannot be contained in a fiber over $w\in W-{\mathcal D}$.

Now let $w\in {\mathcal D}^o$. Then Proposition
\ref{prop:W-and-D} again shows that there is a family of $\PP^1$ covering these
fibers. Now using Proposition \ref{prop:irreducibility}, we get that
the closure is the entire $h^{-1}({\mathcal D})$.
\end{proof}

\section{Torelli theorem}

This section is devoted to a \emph{Torelli type theorem} for the moduli space $M_{\Sp}(L)$, i.e.,
to prove that the moduli space determines the curve $X$ up to isomorphism.

\begin{lemma} \label{lem:4.1}
The global algebraic functions $\Gamma(T^*M_{\Sp}(L))$ produce a map
$$
\widetilde{h}: T^*M_{\Sp}(L) \too \Spec(\Gamma(T^*M_{\Sp}(L)))\cong W
\cong \CC^N\, ,
$$
which is the Hitchin map up to an automorphism of $\CC^N$, where
$N=\dim M_{\Sp}(L)$.

Moreover, consider the standard dilation action of $\CC^*$ on the fibers
of $T^*M_{\Sp}(L)$. Then
there is a unique $\CC^*$-action ``$\cdot$'' on $W$ such that $\widetilde{h}$
is $\CC^*$-equivariant, meaning $\widetilde{h}(E,\lambda\theta)= \lambda
\cdot \widetilde{h}(E,\theta)$.
\end{lemma}

\begin{proof}
This holds for the Hitchin map $H$ on the moduli
of semistable symplectic Higgs bundles $\vicM_{\Sp}(L)$
(cf. \cite{Hi}). On the other hand, the generic fiber
of $H$ is smooth and the codimension
of $T^*M_{\Sp}(L)\subset \vicM_{\Sp}(L)$ on these fibers
is at least two
(cf. \cite[Theorem II.6 (i)]{Fa}, and note that $T^*M_{\Sp}(L)$
is a subset of the moduli $\vicM^0_{\Sp}(L)$ of stable Higgs
bundles). Therefore, it follows that
the lemma also holds
for the restriction of the Hitchin map to the
cotangent bundle $T^*M_{\Sp}(L)$.
\end{proof}

Lemma \ref{lem:4.1} allows us to recover the base $W$ of the Hitchin fibration
as an algebraic manifold. Although there is an isomorphism $W \cong
\bigoplus_{k=1}^n H^0(K_X^{2k})$, we do not know at this point how to
recover the spaces $W_{2k}\subset W$ corresponding to $H^0(K_X^{2k})$.
Lemma \ref{lem:4.1} also gives us the natural ${\mathbb C}^*$-action
on $W$. This gives us the origin of $W$ (as the only fixed
point of the action). Also the subspaces
$W_{\geq 2k}=\bigoplus_{t=k}^n W_{2t}$, for each $t=2,\ldots, n$,
are uniquely determined (these are the spaces where the ratio of
convergence is bigger than $\lambda^{2k}$, for $\lambda \too 0$).
In particular, $W_{2n}\subset W$ is well-determined. Moreover, the $\CC^*$-action on
$W_{2n}$ determines the usual $\CC^*$-action of weight one, which is multiplication
by scalars. This determines the vector space structure of $W_{2n}$.
So we recover $W_{2n}=H^0(K_X^{2n})$.

\begin{proposition}\label{prop:4.2}
Let $\SC$ be the intersection of $W_{2n}=H^0(K_X^{2n}) \subset W$ with
$\SD_1\bigcup \SD_2$. This $\SC$ is irreducible. Moreover
 ${\mathbb P}({\mathcal C})\subset {\mathbb P}(W_{2n})$
 is the dual variety of $X\subset {\mathbb P}(W_{2n}^*)$ for the embedding
 given by the linear series $\vert K_X^{2n}\vert$.
\end{proposition}

\begin{proof}
A spectral curve corresponding to a point of
$s_{2n}\in H^0(K_X^{2n})$
has equation $y^{2n}+s_{2n}(x)=0$, and this curve
is singular at the points with coordinates $(x,0)$
such that $x$ is a zero of $s_{2n}$ of order at least
two. Clearly $\SC=\SD_1\cap W_{2n}$. On the other hand,
$\SD_2\cap W_{2n}\subset \SC$, since it consists of singular curves.
Therefore, the first statement follows.

The elements $b\in W_{2n}$ correspond to spectral curves of the form
$y^{2n}+b(x)=0$.
 We have $b\in {\mathcal C}$ if and only if there is some $x_0$ such
that $b(x_0)=0$ and
$b'(x_0)=0$ simultaneously, therefore
$b\in H^0(K_X^{2n}(-2x_0)) \subset H^0(K_X^{2n})$. From this
the second statement follows, taking into
account that the linear system
$\vert K_X^{2n}\vert $ is very ample, so $X$ is embedded.
\end{proof}

Denote
 $$
 {\mathcal C}_x= H^0(K_X^{2n}(-2x)) \subset W_{2n}
 $$
Then
 $$
 {\mathcal C}= \bigcup_{x\in X} {\mathcal C}_x\, ,
 $$
and taking the bundle of tangent hyperplanes to $X\subset {\mathbb P}(W_{2n}^*)$,
we have
$$
\begin{array}{ccc}
\widetilde{\mathcal C}= \bigsqcup {\mathcal C}_x &
\stackrel{F}{\longrightarrow}
 & {\mathcal C} \\
 \Big\downarrow \\
 X
\end{array}
$$
We shall also need to consider the bundle of hyperplanes through a given
point of $x$, i.e.,
$$
\widetilde{\mathcal H}=\bigsqcup {\mathcal H}_x \too X\, ,
$$
where 
 \begin{equation} \label{eqn:2011-ja}
  {\mathcal H}_x=H^0(K_X^{2n}(-x)) \subset W_{2n} \, .
  \end{equation}
This is intrinsically defined once we have obtained $X$.
\medskip

The following theorem is proved in \cite{BH} by a different method.

\begin{theorem}[Torelli] \label{thm:torelli}
 Let $X$ and $X'$ be two smooth projective curves of genus $g\geq 3$,
and
let
 $M_{\Sp}(L)$ and $M'_{\Sp}(L')$ be moduli spaces of stable symplectic
bundles over $X$ and $X'$ respectively.
 If the variety $M_{\Sp}(L)$ is isomorphic to
$M'_{\Sp}(L')$, then $X\cong X'$.
\end{theorem}

\begin{proof}
Suppose $\Phi: M_{\Sp}(L)\too M'_{\Sp}(L')$ is an isomorphism. Then
there is an isomorphism
 $d\Phi: T^*M_{\Sp}(L)\too T^*M'_{\Sp}(L')$. By Lemma \ref{lem:4.1}, there
is a commutative diagram
$$
\begin{array}{ccc}
T^*M_{\Sp}(L) & \stackrel{d\Phi}{\too} & T^*M'_{\Sp}(L') \\
\Big\downarrow & & \Big\downarrow \\
W & \stackrel{f}{\too} & W'
\end{array}
$$
for some isomorphism $f:W\too W'$. The ${\mathbb C}^*$-actions by
dilations on the fibers of $T^*M_{\Sp}(L)$ and $T^*M'_{\Sp}(L')$
induce ${\mathbb C}^*$-actions on $W$ and $W'$, and $f$ should be
${\mathbb C}^*$-equivariant (as $d\Phi$ is ${\mathbb C}^*$-equivariant).
Therefore $f: W_{2n}\too W_{2n}'$, and it is linear.

We have seen in Proposition \ref{prop:W-and-D}
that the Hitchin discriminant
${\mathcal D}={\mathcal D}_1\cup {\mathcal D}_2\subset W$ is an intrinsically defined subset,
and therefore it is preserved by $f$. So $f$ preserves ${\mathcal C}={\mathcal D}\cap W_{2n}$. This induces
an isomorphism of the corresponding dual varieties, hence by
Proposition \ref{prop:4.2}, an isomorphism $\sigma:X\longrightarrow X'$ is obtained.
\end{proof}

\begin{remark}
{\rm Take $X$ and $X'$ as in Theorem \ref{thm:torelli}. Let
$\overline{M}_{\Sp}(L)$ and $\overline{M}'_{\Sp}(L')$ be moduli spaces
of semistable symplectic bundles over $X$ and $X'$ respectively.
The proof of Theorem \ref{thm:torelli} gives that
$X$ is isomorphic to $X'$ if $\overline{M}_{\Sp}(L) \cong
\overline{M}'_{\Sp}(L')$.}
\end{remark}

\begin{remark}\label{rem:2011-1}
{\rm Let us see that we can extend the previous arguments to recover the
(linear)
projection $\pi_{2n}:W \longrightarrow W_{2n}$ (even though we do not
know the linear structure of $W$). This means that the map
$f$ in the proof of Theorem \ref{thm:torelli} commutes with $\pi_{2n}$.

First of all, the two irreducible components ${\mathcal D}_1$, ${\mathcal D}_2$
of ${\mathcal D}$ are well characterized. Suppose this for a moment.
The divisor ${\mathcal D}_1$ consists
of $(s_2,\ldots, s_{2n})$ such that there exists a point $x\in X$ with
$s_{2n}(x)=s_{2n}'(x)=0$.
So, writing $W'=\bigoplus_{k=1}^{n-1} W_{2k}\subset W$, we have
 $$
  {\mathcal D}_1 = \bigcup_{x\in X} (W'\times {\mathcal C}_x ) \subset W\, .
 $$
As such, there is a unique rational map ${\mathcal D}_1 \dashrightarrow X$ (up to
automorphisms of the base) whose fibers are connected and rational (actually,
isomorphic to open subsets of linear spaces). This determines each
fiber $W'\times {\mathcal C}_x$, $x\in X$, and a fortiori, the intersection
$W'\subset W$ of all of them. It is not difficult to see
that the
information of $W_{2n}$, $W'$ and the $\CC^*$-action together give
the decomposition $W=W'\times W_{2n}$. The projection $\pi_{2n}$ follows.

It remains to identify the subvariety ${\mathcal D}_1$. Linearize the
$\CC^*$-action at the origin of $W$. According to the weights,
the tangent space $T_0W$ decomposes as $\bigoplus_{k=1}^n W_{2k}$ (here the
decomposition is clearly well-defined), and also
$W'=\bigoplus_{k=1}^{n-1} W_{2k}
\subset T_0W$. The tangent cone to ${\mathcal D}_1$ is
$\bigcup_{x\in X} ( W'\times {\mathcal C}_x) \subset T_0W$. This is \emph{not} the
tangent cone to ${\mathcal D}_2$. So this property (for instance) can be used
to characterize ${\mathcal D}_1$.}
\end{remark}

\section{Nilpotent cone and flag variety}

We need some results on linear algebra about the space of symplectic
endomorphisms.
Let $(V,\omega)$ be a symplectic vector space of dimension $2n$.
In this section, we want to study the
\textit{symplectic nilpotent cone}
 $$
 {\mathcal N} = \{ A\in \End_{\Sp} V \, \mid \, A^{2n}=0\}\, .
 $$

\begin{lemma}\label{lem:N}
 The following statements hold:
 \begin{enumerate}
 \item ${\mathcal N}$ is a $2n^2$-dimensional algebraic variety.
 \item $A\in {\mathcal N}$ if and only if ${\rm tr}(A^{2r})=0$, for all
$r= 1, \ldots, n$.
 \item $A\in {\mathcal N}^{sm}$ (the smooth locus of $\mathcal N$) if
and only if $\rk A=2n-1$.
 \item Let $Fl(V,\omega)$ be the set of full isotropic flags on ${\mathbb C}^{2n}$. Then there is
 a fibration $\pi:{\mathcal N}^{sm} \too Fl(V,\omega)$ with fibers
isomorphic to
 $({\mathbb C}^*)^{n} \times {\mathbb C}^{n^2-n}$.
 \end{enumerate}
\end{lemma}

\begin{proof}
 Statement (1) is clear, since ${\mathcal N}$ is defined by the
equations
$q_2(A)=\ldots =q_{2n}(A)=0$, where $p_A(t)=t^{2n}+q_2(A) t^{2n-2} + \ldots + q_{2n}(A)$ is
the characteristic polynomial of $A\in \End_{\Sp}V$.
As $\dim \End_{\Sp}V=n(2n+1)$, and we have $n$ equations,
it follows that $\dim {\mathcal N}=2n^2$.

To prove statement (2), note that if $\pm \lambda_1,\ldots, \pm
\lambda_n$ are the eigenvalues of $A$,
then ${\rm tr}(A^r)= 0$ for $r$ odd, and ${\rm tr}(A^r)=2\sum
\lambda_i^r$ for $r$ even.
Then the equations ${\rm tr}(A^r)=0$, $r=2,4,\ldots, 2n$, are
equivalent to $\lambda_1=\ldots =\lambda_n=0$, i.e., $A^{2n}=0$.

Now we will prove statement (3). Let $B\in \End_{\Sp}V$. Considering
$\frac{d}{d\epsilon} \big\vert_{\epsilon=0} {\rm tr}((A+\epsilon
B)^{2r})=0$, $r=1,\ldots, n$,
we see that
 $$
 T_A{\mathcal N}=\{B\, \mid \, {\rm tr}(A^{2r-1} B)=0, \, r=1,\ldots,
n\}\, .
 $$
This has codimension strictly less than $n$ when $A^{2n-1}=0$. So $\rk
A=2n-1$ at
a smooth point.
For the converse, if $\rk A=2n-1$, then the matrices $I, A, A^2, \ldots, A^{2n-1}$
are linearly independent, and $A^{2k-1}\in \End_{\Sp}V$, $k=1,\ldots, n$.
Therefore, the $n$ equations ${\rm tr}(A^{2r-1}B)=0$, $r=1,\ldots, n$,
for $B\in \End_{\Sp}V$ are
linearly independent, and $\codim T_A{\mathcal N}=n$. Hence $A\in {\mathcal N}^{sm}$, as required.

Finally we prove statement (4). Note that if $A\in {\mathcal N}^{sm}$,
then $\rk A=2n-1$.
This determines a well-defined full flag
 \begin{equation}\label{eqn:flag}
 0 \subset \ker A \subset \ker A^2 \subset \ldots \subset \ker A^{2n-1}
\subset {\mathbb C}^{2n}\, .
 \end{equation}
Let us see that $\ker A^i$ is dual (with respect $\omega$) to $\ker A^{2n-i}$. For this, note that
$\ker A^{2n-i}=\im A^i$. If $u=A^i u_0$, and $v\in \ker A^i$, then
$$\omega(u,v)=\omega(A^iu_0,v)=
(-1)^i \omega (u_0, A^i v)=0\, .$$ This means that the flag
in (\ref{eqn:flag}) is isotropic (in particular,
$\ker A^n$ is Lagrangian).

The fiber over a point of the flag variety is given as follows. Fix a
symplectic basis $e_1,\ldots, e_n,e_{n+1},
\ldots, e_{2n}$, such that the flag is
 \begin{equation}\label{eqn:f0}
 \langle e_{n} \rangle \subset \langle e_{n-1},e_{n}\rangle \subset
\cdots \subset
\langle e_{1},\ldots, e_{n} \rangle \subset
\langle e_1, \ldots, e_{n},e_{n+1} \rangle \subset
\cdots \subset \langle e_1, \ldots, e_{2n} \rangle \, .
 \end{equation}
Then the matrices in the fiber are of the form
\begin{equation}\label{eqn:jordan}
\left(\begin{array}{cc} A & B \\ 0 & -A^T
\end{array}\right) , \text{ where }
A= \left(\begin{array} {ccccc}
0 & 0& 0&\ldots & 0 \\
a_{21} & 0 &0& \ldots & 0\\
a_{31}& a_{32}& 0 & \ldots & 0 \\
\vdots & &\ddots&\ddots & \vdots \\
a_{n,1} & a_{n,2} & \ldots & a_{n, n-1} & 0
 \end{array} \right)\,
 \end{equation}
with $a_{i+1,i}\neq 0$, and $B=(b_{ij})$ is symmetric with $b_{11}\neq 0$.
So the fiber of
 $$\pi:{\mathcal N}^{sm} \too Fl(V,\omega)$$
is
$({\mathbb C}^*)^{n} \times {\mathbb C}^{n^2-n}$ (recall that
$Fl(V,\omega)$ is defined in the statement (4) of the lemma).
\end{proof}

We will now prove that ${\mathcal
N}^{sm}$ determines $Fl(V,\omega)$.
Consider the fibration $\pi$ in Lemma \ref{lem:N} (4).
We shall show that the fibers of $\pi$ are intrinsically defined. Take any
$F\in Fl(V,\omega)$. Let $$U_F\subset {\mathcal N}\subset \End_{\Sp}(V)$$
be the space
of symmetric endomorphisms respecting the flag $F$. It is a
linear subspace of $\End_{\Sp}(V)$ of dimension $n^2$ contained in the
nilpotent cone. By Lemma \ref{lem:agotador}
below, all subspaces of $\End_{\Sp}(V)$ of dimension $n^2$ contained in
the nilpotent cone are of the form $U_F$ for some $F$. Moreover, $U_F\cap
{\mathcal N}^{sm}$ is the fiber of $\pi$ over $F$. Therefore
$\pi$ is uniquely defined, up to
automorphism of the base.

\begin{lemma} \label{lem:agotador}
Let $L \subset {\mathcal N}\subset \End_{\Sp}(V)$ be a linear subspace
of dimension $n^2$ such that $L \cap {\mathcal N}^{sm} \neq
\emptyset$. Then there exists a unique flag $F$ such that $L=U_F$.
\end{lemma}

\begin{proof}
We shall divide the proof in several steps.

\noindent \textbf{Step 1.} $\tr (A^i B)=0$, for any $A,B\in L$, $i\geq 0$.
If $i$ is even, then $A^iB$ is an anti-symmetric endomorphism, hence of zero trace.
For $i$ odd, note that $\tr(C^{2j})=0$ for any $C\in {\mathcal N}$, $j\geq 0$. Then considering
that $C_\lambda=A+\lambda B\in L\subset {\mathcal N}$, we have that
$\tr(C_\lambda^{i+1})=\tr ((A+\lambda B)^{i+1})=0$. Take the coefficient of
$\lambda$ to get $\tr (A^i B)=0$.
\medskip

\noindent \textbf{Step 2.} Let $A\in {\mathcal N}$. Then $A\in L$ if and only if $\tr(AB)=0$, for
all $B\in L$.

 The ``only if'' part follows from Step 1.

To prove the ``if'' part, suppose that $\tr (AB)=0$, for all $B\in L$.
 As $A\in {\mathcal N}$, we can choose a flag so that $A \in U$ (this is unique if $A\in {\mathcal N}^{sm}$).
 Taking an appropriate symplectic basis, the flag is the one in
 (\ref{eqn:f0}), and there is a well-defined space $U$ of symplectic
symmetric matrices (see \eqref{eqn:jordan})
 preserving the flag.

 Denote $U^T=\{B^T \, \mid \, B\in U\}\subset \End_{\Sp}(V)$. Consider
also
 the space $D\subset \End_{\Sp}(V)$ consisting of symmetric symplectic
 diagonal matrices. Therefore
$$
\End_{\Sp}(V)= U \oplus D \oplus U^T\, .
$$
 The bilinear map $q(B_1,B_2)=\tr(B_1B_2)$ is symmetric and non-degenerate.
 The $q$-dual of $U$ is $U+D$. As $\End_{\Sp}(V)/(U+D)=U^T$,
 there is an induced perfect pairing $q:U\times U^T \too {\mathbb C}$.
 Let $p=p_{U^T}: \End_{\Sp}(V) \too U^T$ be the projection. Now $L\cap
(U+D) =L\cap U$,
 since all elements of $L$ are nilpotent. Let us see that $L\cap U$ is $q$-dual to $p(L)$.
 Clearly, $q(B,C_1)=\tr(BC_1)=\tr(BC)=0$, for $B\in L\cap U$, $C_1\in
p(L)$ and $C=C_1+C_2\in L$, with $C_2\in U+D$.
 On the other hand,
$$
\dim (L\cap U) + \dim p(L) =\dim L=\dim U\, .
$$
 Therefore, $p(L) \subset U^T$ and $L\cap U \subset U$
 are $q$-orthogonal complements.

 The conclusion is that given $B\in U$, we have
 $B\in L\cap U$ if and only if $q(B,C_1)=0$ for all $C_1\in p(L)$.
 In short, $B\in L\cap U$ if and only if $ \tr(BC)=0$, for all $C\in L$.

\medskip
\noindent \textbf{Step 3.} If $A\in L$ then $A^{2i-1}\in L$ for any $i\geq 1$.
 Without loss of generality, we can suppose $A\in {\mathcal N}^{sm}$.
Therefore, $A$ determines a
 flag $F$, and a space $U=U_F$ of endomorphisms preserving the flag.
 By Step 1, $\tr(A^{2i-1}B)=0$, for any $B\in L$. By Step 2,
we have $A^{2i-1}\in L$.
\medskip

 \noindent \textbf{Step 4.} Let $A,B\in L$, then $A^2B+BA^2+ABA\in L$.
 Let $C_\lambda=A+\lambda B\in L$. Then $C_\lambda^3=(A+\lambda B)^3
 \in L$ by Step 3. Take the coefficient of $\lambda$, to conclude the
statement.
\medskip

\noindent \textbf{Step 5.} Let $A\in L \cap {\mathcal N}^{sm}$.
Let $F$ be the (isotropic) flag determined by $A$, and let $U=U_F$.
Denote as $0\subset V_1\subset V_2 \subset \cdots \subset V_{2n}=V$ the flag $F$.
For $B\in L$, let $r(B)\in {\mathbb Z}$ be the minimum integer such that
$B(V_i)\subset V_{i+r}$. (Note that if $r(B)<0$ is equivalent to $B\in U$.)
We claim that either $r(B)<0$ or $r(B)$ is odd.

Let $r=r(B)$. If $r=0$, then $B\in U+D$. As $B$ is nilpotent,
it follows that $B\,\in\, U$, meaning $r<0$. Now we work by
induction on $r$.
Suppose that $r>0$ and it is even.
Write $B=(b_{ij})$, in some basis adapted to the flag, 
and note that $b_{ij}=0$ for $i-j>r$. Let
$b_i=b_{r+i,i}$, $i=1,\ldots, 2n-r$.
By Step 4, we have $C=A^2B+BA^2+ABA\in L$. It is easy to see that
$r(C)\leq r-2$, and
that it has coefficients
 \begin{eqnarray*}
 && c_1=b_1, c_2=b_1+b_2, c_3=b_1+b_2+b_3 , \ldots, c_i=b_{i-2}+b_{i-1}+b_i, \\
 && \ldots, c_{2n-r+1}=b_{2n-r-1} +b_{2n-r},c_{2n-r+2}=b_{2n-r}\, .
 \end{eqnarray*}
 By induction hypothesis, $r(C)\neq r-2$, so $r(C)<r-2$ and all $c_i=0$. From here,
 $b_i=0$ for all $i$, and so $r(B)<r$.

\medskip

\noindent \textbf{Step 6.} With the notation as in Step 5, $r=r(B)>1$.
Suppose that $r=1$. Let $\{v_i\}$ be a basis adapted to the flag, i.e.
$V_i=\langle v_1,\ldots, v_i\rangle$,
for all $i=0,\ldots, 2n$, and consider the coefficients $b_i:=b_{i+1,i}$,
$i=1, \ldots, 2n-1$, not all equal to zero.

Suppose first that all $b_i\neq 0$. Then $B$ has rank $2n-1$, so $\ker B$ is $1$-dimensional.
Actually, if $v$ spans $\ker B$, then $v\notin V_{2n-1}$. So choose the basis so that
$v_{2n}=v$ and $v_k=A(v_{k+1})$, $k=1,\ldots, 2n-1$. Therefore $A$ has standard Jordan form
and $b_{j,2n}=0$, all $j$.
So $\det(B+\lambda A)=b_{2n-1}\lambda \det(B'+\lambda A')$, where $A', B'$ are $(2n-2)\times
(2n-2)$-matrices obtained from $A,B$ by removing the last two columns and rows. By induction, this determinant is
non-zero. Therefore $A+\lambda B$ is not nilpotent for generic value of $\lambda$. This is a contradiction,
since $A+\lambda B\in L\subset {\mathcal N}$.

Now suppose that $b_{i_0}=0$, $b_{i_0+2k}=0$, but $b_{i_0+1},\ldots, b_{i_0+2k-1}\neq 0$. Then take
the blocks formed by rows and columns $i_0+1,\ldots, i_0+2k$. This produces
matrices $A',B'$ of even size, such that $A'+\lambda B'$ is nilpotent for all $\lambda$. But
$\det(A'+\lambda B')\neq 0$, which is proved as before.

The next case is that $b_{i_0}=0$, $b_{i_0+2k+1}=0$, but $b_{i_0+1},\ldots, b_{i_0+2k}\neq 0$.
Choose one such possibility with the smallest possible value of $k$.
Let $W\subset {\mathbb C}^{2n-1}$ be the subspace parametrizing vectors $(b_i)$ arising from matrices
$B\in L$ with $r(B)=1$. And let $W_{i_0,2k+1}$ be the subspace of those vectors $(b_{i_0+1},\ldots, b_{i_0+2k})$
where $(b_i)\in W$, $b_{i_0}=0$, $b_{i_0+2k+1}=0$. If this has dimension $\geq 2$, then there is a vector
with some coordinate zero. Therefore, there is a smaller $k$. So $\dim W_{i_0,2k+1}=1$.
Step 4 implies that if $(b_{i_0+1},\ldots, b_{i_0+2k})\in W_{i_0,2k+1}$ then
 \begin{eqnarray*}
 && (b_{i_0+1}(b_{i_0+1}+b_{i_0+2}),b_{i_0+2}(b_{i_0+1}+b_{i_0+2}+b_{i_0+3}), \ldots, \\ && \qquad
 b_{i_0+2k-1}(b_{i_0+2k-2}+b_{i_0+2k-1}+b_{i_0+2k}),b_{i_0+2k}(b_{i_0+2k-1}+b_{i_0+2k}) )
 \in W_{i_0,2k+1}\, .
 \end{eqnarray*}
As this vector is a multiple of the previous one, it must be
 $$
 b_{i_0+1}+b_{i_0+2}= b_{i_0+1}+b_{i_0+2}+ b_{i_0+3}= \ldots =b_{i_0+2k-1}+b_{i_0+2k}.
 $$
This implies the vanishing of all $b_i$ unless $k=1$.
And moreover, if $k=1$, then taking the $3\times 3$-matrix with rows and
columns $i_0+1, i_0+2,i_0+3$, we get that $b_{i_0+1}=\alpha$,
$b_{i_0+2}=-\alpha$, for some $\alpha\in {\mathbb C}$, by using that
$B'+\lambda A'$ should be nilpotent.

So the elements of $W$ are of the form $(\ldots, 0, \alpha_1, -\alpha_1, 0, \ldots,
0, \alpha_2, -\alpha_2, 0,\ldots)$.

Now consider the space $H^T:=\{ B \in U^T \,\mid\, r(B)\leq 2\}$. The
dual of $H^T$ under $q$
is denoted $Z \subset U$, and let $H:=U / Z$, with projection $p_H: U\too
H$.
So there is a perfect pairing $q:H^T \times H \too {\mathbb C}$.
Recall that Step 2 says that $U\cap L$ is $q$-dual to $p_{U^T}(L)$. Therefore
$p_{U^T}(L) \cap H^T$ and $p_H(U\cap L)$ are $q$-dual.
Now consider $B_2\in p_{U^T}(L) \cap W^T$. Then
$B=B_1+B_2\in L$, where $B_2\in U+C$, and $r(B)\leq 2$. By Step 5, $r(B)\leq 1$.
By the previous discussion, if $r(B)=1$, then the entries $(b_i)\in W$ have the form
given above.

So $p_H(U\cap L)$ contains matrices with any values in the second diagonal,
and with values in the first diagonal of the form $(\ldots, *, \beta_1, \beta_1, *, \ldots,
*, \beta_2, \beta_2, *,\ldots)$.

Now just consider a matrix $C\in U\bigcap L$
with all zeroes on the first diagonal, and just one $1$
in the second diagonal, in a position $(i_0,i_0+2)$, so that the $3\times 3$-block
coming from $B+\lambda A+C\in L$ has a $1$ in the right-top corner. This matrix is not
nilpotent. This is a contradiction.
\medskip

\noindent \textbf{Step 7.} $L=U$. Fix some $A\in L\cap {\mathcal N}^{sm}$,
and the corresponding subspace $U$. Let $B\in L$, and let $r=r(B)$. We only
have to see that $r<0$. If $r\geq 0$, then $r$ is odd from Step 5. By Step 6,
it cannot be $r=1$. The same argument as in Step 5, proves that it cannot
be $r>1$ and odd. So $B\in U$. Hence $L\subset U$, so they are equal by dimensionality.
\end{proof}

\begin{remark}
{\rm Lemma \ref{lem:agotador} also follows from the main theorem
in \cite{DKK}. The above proof, which uses only elementary
methods, is included in order to be self-contained.}
\end{remark}

\section{Automorphisms of the moduli space}

In this section we will compute the automorphism group of a moduli
spaces of symplectic bundles on $X$. As before, we assume that
$g\geq 3$.

\begin{proposition} \label{prop:lll}
Fix a generic stable bundle $E\in M_{\Sp}(L)$, and consider the map
 $$
 h_{2n}: H^0(\End_{\Sp} E\otimes K_X) \too W_{2n} \, ,
 $$
given as composition of the Hitchin map on $H^0(\End_{\Sp} E\otimes K_X)=T^*_EM_{\Sp}(L)\subset T^*M_{\Sp}(L)$,
followed by projection $\pi_{2n}:W\to W_{2n}$ (see Remark \ref{rem:2011-1}).
Then
 $$
 H^0(\End_{\Sp} E\otimes K_X(-x_0))
$$
$$
= \big\{ \psi \in H^0(\End_{\Sp} E\otimes K_X) \, \mid \,
h_{2n}(\psi+\phi)
 \in {\mathcal H}_{x_0}, \,\forall\, \phi\in h_{2n}^{-1}({\mathcal
H}_{x_0}) \big\}\, ,
$$
where ${\mathcal H}_x$ is defined in (\ref{eqn:2011-ja}).
\end{proposition}

\begin{proof}
First, note that the sequence
 $$
 0\too H^0(\End_{\Sp} E\otimes K_X(-x_0)) \too H^0(\End_{\Sp} E\otimes K_X)
\too \End_{\Sp} E\otimes K_X\vert_{x_0} \too 0
 $$
is exact, since $H^1(\End_{\Sp} E\otimes K_X(-x_0))=H^0(\End_{\Sp} E (x_0))^*= 0$, for a {generic} bundle,
by Lemma \ref{lem:asdfg2}.
So the map
 $$
 H^0(\End_{\Sp} E\otimes K_X) \too \End_{\Sp} E\otimes K_X\vert_{x_0}\,
,
 $$
given by $\phi \longmapsto \phi(x_0)$, is surjective.

Note that $h_{2n}(\phi) =\det(\phi) \in W_{2n}=H^0(K_X^{2n})$. So
$$h_{2n}(\phi)\in {\mathcal H}_{x_0}
\iff \det(\phi(x_0))=0\, .$$ The result follows from this easy linear
algebra fact:
if $(V,\omega)$ is a symplectic vector space, and $A\in \End_{\Sp}(V)$ satisfies that $\det(A+C)=0$ for any
$C\in \End_{\Sp}(V)$ with $\det (C)=0$, then $A=0$.
\end{proof}

Proposition \ref{prop:lll} allows to construct the bundle
 $$
 \SE \too X
 $$
whose fiber over $x\in X$ is $\SE_x = H^0(\End_{\Sp} E \otimes K_X(-x))$.
This is a
subbundle of the trivial bundle
 $$
 H^0(\End_{\Sp}E\otimes K_X) \otimes {\mathcal O}_X \longrightarrow X\, ,
 $$
and there is an exact sequence
 \begin{equation}\label{eqn:28aug}
 0 \too \SE \too H^0(\End_{\Sp}E \otimes K_X) \otimes {\mathcal O}_X
\stackrel{\pi}{\longrightarrow}
 \End_{\Sp} E\otimes K_X\too 0 \, .
 \end{equation}
This is exact on the right by Lemma \ref{lem:asdfg2}.
So we recover the bundle
 $$
 \End_{\Sp} E\otimes K_X\too X\, .
 $$

\begin{remark}
{\rm A natural continuation of argument is as follows.
Take an automorphism $\Phi:M_{\Sp}(L)\too M_{\Sp}(L)$ of the
moduli space, let $E\in M_{\Sp}(L)$ be a generic bundle,
and let $E'=\Phi(E)$. Then, by the previous argument,
there is a bundle isomorphism $\End_{\Sp} E \cong \End_{\Sp} E'$. From here 
we would like to deduce that $E'$ is the twist of $E$ by a line bundle. However,
this is not known (at least, to the authors). This would 
amount to proving the generic injectivity of the map
$E\mapsto \End_{\Sp} E$ between the corresponding moduli spaces.

Because of this, we shall take an alternative route, determining first
the nilpotent symplectic cones bundle, then the isotropic flag varieties bundle,
and from this the Lie algebra structure of $\End_{\Sp} E$. This actually
determines $E$ up to twist by a line bundle.}
\end{remark}

\medskip

We start with the following lemma:

\begin{lemma} \label{lem:2011-2}
The linear structure of the base of the Hitchin map $W$ is uniquely
determined. So is the decomposition $W=\bigoplus_{k=1}^n W_{2k}$, where
$W_{2k}=H^0(K_X^{2k})$, $k=1,\ldots,n$.
\end{lemma}

\begin{proof}
 Let $M_{\Sp}(L)$ be the moduli space of symplectic bundles, and let
 $\Phi:M_{\Sp}(L)\too M_{\Sp}(L)$ be an automorphism.
 As in the proof of Theorem \ref{thm:torelli}, the automorphism $\Phi$
 yields an isomorphism $\sigma:X\too X$ and commutative diagrams
 \begin{equation}\label{eqn:2011-3}
 \begin{array}{ccccccc}
 \widetilde{\mathcal C} & {\too} &
\widetilde{\mathcal C} & \qquad \text{and} \qquad &
 \widetilde{\mathcal H} & {\too} &
\widetilde{\mathcal H} \\
 \Big\downarrow & & \Big\downarrow &&
 \Big\downarrow & & \Big\downarrow \\ X & \stackrel{\sigma}{\too}
& X
&&
X & \stackrel{\sigma}{\too} & X
 \end{array}
 \end{equation}
Let $M$ be a line bundle such that $M^{\otimes 2}\cong L\otimes (\sigma^*L)^\vee$.
 Composing with the automorphism given by $\sigma^{-1}$ and $M^{-1}$
(see the introduction),
we may assume that $\sigma=Id$.
 Take a \emph{generic} bundle $E$,
 and let $E'$ be its image by $\Phi$. Then we have
 \begin{equation}\label{eqn:2011-4}
  \begin{array}{ccc}
  T^*_EM_{\Sp}(L) & \stackrel{d\Phi}{\too} & T^*_{E'}M_{\Sp}(L) \\
  h \Big\downarrow \  & & h \Big\downarrow \ \\
  W & \stackrel{f}{\too} & W
  \end{array}
 \end{equation}
where $f$ is an automorphism which commutes with
the ${\mathbb C}^*$-action. To prove the lemma, we have to check
that $f$ is linear with respect to a chosen decomposition $W=\oplus_{k=1}^n W_{2k}$.

There is a bundle $\SE \too X$ whose fiber
over any $x\in X$ is the subspace
 $$
 H^0(\End_{\Sp} E\otimes K_X(-x))\subset T_E^*M_{\Sp}(L)\, .
 $$
Analogously, there
is a bundle $\SE'\too X$ whose fiber over $x$ is
the subspace
 $$H^0(\End_{\Sp} E'\otimes K_X(-x))\subset T_{E'}^*M_{\Sp}(L)\, .$$
By Proposition \ref{prop:lll},
the map $d\Phi: T_E^*M_{\Sp}(L) \too T_{E'}^*M_{\Sp}(L)$ gives an
isomorphism $\SE \too \SE'$.
Going to the quotient bundle (\ref{eqn:28aug}), we have a bundle isomorphism
 $$
 \End_{\Sp} E\otimes K_X \longrightarrow \End_{\Sp} E'\otimes K_X\, .
 $$
Let
 $$
 \varphi\,:\, \End_{\Sp} E \,\longrightarrow \,\End_{\Sp} E'
 $$
be the corresponding isomorphism. Then
the map 
 $$
 d\Phi: H^0(\End_{\Sp} E\otimes K_X) \longrightarrow
 H^0(\End_{\Sp} E' \otimes K_X)$$
is induced by $\varphi$.

The Hitchin map $h:H^0(\End_{\Sp} E\otimes K_X) \longrightarrow W$
factors as the composition of the map
 $$
 H^0(\End_{\Sp} E\otimes K_X) \longrightarrow \prod_{k=1}^n
 H^0(\bigwedge\nolimits^{2k} \End_{\Sp} E \otimes K^{2k}_X)
 $$
and the trace map
 $$
 \prod_{k=1}^n
 H^0(\bigwedge\nolimits^{2k} \End_{\Sp} E \otimes K^{2k}_X) \longrightarrow
 \prod_{k=1}^n W_{2k}\, .
 $$

By the discussion above, the map $d\Phi$ descends to a map
 $$\prod_{k=1}^n
 H^0(\bigwedge\nolimits^{2k} \End_{\Sp} E \otimes K^{2k}_X) \longrightarrow
 \prod_{k=1}^n
 H^0(\bigwedge\nolimits^{2k} \End_{\Sp} E' \otimes K^{2k}_X)
 $$ 
and is diagonal
 with respect to the product decomposition. Therefore, the map
 $f:\prod_{k=1}^n W_{2k} \to \prod_{k=1}^n W_{2k}$ is also
 diagonal.
 Finally, this means that we have maps $f:W_{2k}\to W_{2k}$ which
 are $\CC^*$-equivariant. As there is just one weight for $W_{2k}$,
 this means that $f$ is linear on $W_{2k}$. Therefore $f$ is linear.
\end{proof}

Our next step is to recover the symplectic nilpotent cone bundle
 $$
 {\mathcal N}_E \too X \, ,
 $$
whose fibers are the symplectic nilpotent cone spaces
 $$
 {\mathcal N}_{E,x} =\{ A \in \End_{\Sp}E\otimes K_X\vert_{x} \, \,
\mid \,\, A^{2n}=0\}
\subset \End_{\Sp}E\otimes K_X\vert_{x},  \ x \in X\, .
 $$
Note that this nilpotent cone bundle sits as
${\mathcal N}_E \subset \End_{\Sp} E\otimes K_X$.

\begin{lemma} \label{lem:last-step}
 Consider the map
 $$
 h_{2k}: H^0(\End_{\Sp} E\otimes K_X) \too W_{2k} \, ,
 $$
given as composition of the Hitchin map on $H^0(\End_{\Sp} E\otimes K_X)=T^*_EM_{\Sp}(L)\subset T^*M_{\Sp}(L)$,
followed by the projection $W\to W_{2k}$ (defined thanks to Lemma
\ref{lem:2011-2}).
Then the vector subspace generated by the image $h_{2k}(H^0(\End_{\Sp} E\otimes K_X(-x)))$
is $$H^0(K_X^{2k}(-2k\, x)) \subset W_{2k}= H^0(K_X^{2k})\, .$$
\end{lemma}

\begin{proof}
 The map $h_{2k}$ sends $(E,\phi)\mapsto \tr (\wedge^{2k}\phi)$.
 Therefore, $$h_{2k}(H^0(\End_{\Sp} E\otimes K_X(-x)))
\, \subset\, H^0(K^{2k}_X(-2k\, x))\, .$$

 Now the vector space generated by $h_{2k}(H^0(\End_{\Sp} E\otimes K_X(-x)))$
 equals to the image of
 $$
 \wedge^{2k}:H^0(\End_{\Sp} E\otimes K_X(-x))^{\otimes 2k} \too
 H^0(\End_{\rm ant} E\otimes K_X^{2k}(-2k\, x)) \, ,
 $$
 followed by the trace map
 $H^0(\End_{\rm ant} E\otimes K_X^{2k}(-2k\, x)) \too H^0(K^{2k}_X(-2k\, x))$,
 where $\End_{\rm ant} E$ consists of the anti-symmetric symplectic endomorphisms of $E$ (i.e., those
 $\varphi$ such that $\omega( \varphi \, u, v)=\omega(u,\varphi \, v)$).
 Note that the multiples of the identity are in $\End_{\rm ant} E$.

 There is a bundle map
 $$
 (\End_{\Sp} E)^{\otimes 2k} \too  \End_{\rm ant} E \too {\mathcal O}_X
 $$
 (first map is composition of endomorphisms, second map is the trace).
 This is split, therefore the map
 $$
 H^0((\End_{\Sp} E)^{\otimes 2k} \otimes K_X^{2k}(-2k\, x)) \too  H^0(K_X^{2k}(-2k\, x))
 $$
 is surjective, as required.
\end{proof}

Consider the vector subspace generated by the image $h_{2k}(H^0(\End_{\Sp} E\otimes K_X(-x)))$
which is $H^0(K^{2k}_X(-2k\, x)) \subset W_{2k}= H^0(K^{2k}_X)$, for
moving the point $x\in X$.
This gives a fiber bundle
over $X$, which has co-rank $2k$. The curve $X$ is embedded
into $\PP (W_{2k}^*)$ via the linear system $\vert K_X^{2k}\vert$, and the
osculating $2k$-space at $x$ is given  as
 $$
 Osc_{2k}(x)=\PP(V_x) \subset \PP (W_{2k}^*),
 \qquad V_x:=\ker (H^0(K^{2k}_X)^* \to H^0(K^{2k}_X(-2k\, x)^*))\, .
 $$
The embedding of the curve $X$ in $\PP(W_{2k}^*)$ is recovered from the osculating $2k$-spaces.
More specifically, if $g:X\to \PP^N$ and $\tilde{g}: X\to \Gr(2k+1,N+1)$ is the
map giving the osculating $2k$-spaces, then $\tilde{g}$ determines $g$. This is
proved as follows: the pull-back of the universal bundle through $\tilde{g}$ is
the bundle
 $$
 \PP ({\mathcal O}_X \oplus TX \oplus (TX)^{\otimes 2} \oplus \ldots \oplus (TX)^{\otimes 2k}) \too X\, ,
 $$
and $g$ is determined by a section of this bundle. As $TX$ is of
negative degree, this has only one section.

So we recover the embeddings $X\hookrightarrow \PP(W_{2k}^*)=\PP (H^0(K^{2k}_X)^*)$, and hence the hyperplanes
 $$
 H_x^{(2k)} := H^0(K^{2k}_X(-x)) \subset W_{2k}\, .
 $$

Finally, consider
 $$
 \{\phi \in H^0(\End_{\Sp}E \otimes K_X) \, \mid \, h_{2k}(\phi)\in
H_x^{(2k)}, \,\forall\, k=1,\ldots, n \}\, .
 $$
This is the pre-image of the nilpotent cone under the surjective map
$$H^0(\End_{\Sp}E \otimes K_X) \too \End_{\Sp}E \otimes K_X\vert_x\, .$$
Take its image to get the bundle
 $$
 {\mathcal N}_E \too X \, .
 $$
\medskip

Now we are ready to prove the main result of the paper. First,
as explained in the introduction, line bundles of order two and automorphisms of $X$ produce
automorphisms of moduli spaces of symplectic bundles.

\begin{theorem} \label{thm:autom}
 Let $M_{\Sp}(L)$ be the moduli space of symplectic bundles. Let
$$\Phi:M_{\Sp}(L)\too M_{\Sp}(L)$$
 be an automorphism.
 Then $\Phi$ is induced by an automorphism of $X$ and a line bundle of order two.
\end{theorem}

\begin{proof}
 We work as in Lemma \ref{lem:2011-2}. There is an
 automorphism $\sigma:X\too X$ and a line bundle
 $M$, with $M^{\otimes 2}\cong L\otimes (\sigma^*L)^\vee$,
 such that, after composing with $\sigma^{-1}$ and $M^{-1}$
(as in the introduction), we have the diagram
 (\ref{eqn:2011-3}) with $\sigma=Id$.

 Take a \emph{generic} bundle $E$, and let $E'$ be its image by $\Phi$.
 By the arguments in Lemma \ref{lem:2011-2},
 we have a bundle isomorphism
 $$
 \varphi: \End_{\Sp} E \longrightarrow \End_{\Sp} E'\, ,
 $$
and the diagram (\ref{eqn:2011-4}) becomes
 \begin{equation*}
  \begin{array}{ccc}
  H^0( \End_{\Sp} E\otimes K_X) & \stackrel{d\Phi}{\too} & H^0( \End_{\Sp} E'\otimes K_X) \\
  h \Big\downarrow \  & & h \Big\downarrow \ \\
  W & \stackrel{f}{\too} & W
  \end{array}
 \end{equation*}
where $h$ is the Hitchin map and $f$ is a linear isomorphism.
By Lemma \ref{lem:last-step} and the discussion following it, $f$
preserves  $H_x^{(2k)} = H^0(K^{2k}_X(-x)) \subset W_{2k}$,
for each $k=1,\ldots, n$, and $x\in X$.

Therefore $d\Phi$ preserves $\bigcap_{k\geq 1} h_{2k}^{-1}(H_x^{(2k)}) \subset
 H^0(\End_{\Sp}E \otimes K_X)$. Taking the image under the surjective map
$H^0(\End_{\Sp}E \otimes K_X) \too \End_{\Sp}E \otimes K_X\vert_x$, we
see that
$d\Phi$ preserves the symplectic nilpotent cone ${\mathcal N}_{E,x}$, for all $x\in X$.
So we have an isomorphism
 $$
 \begin{array}{ccc}
 {\mathcal N}_E & {\too} & {\mathcal N}_{E'} \\
 \Big\downarrow & & \Big\downarrow \\
 X & = & X
 \end{array}
 $$
where ${\mathcal N}_E$, ${\mathcal N}_{E'}$ are the corresponding symplectic nilpotent cone bundles.

By Lemma \ref{lem:agotador}, we get an isomorphism
 $$
 \begin{array}{ccc}
 Fl(E,\omega) & {\too} & Fl({E'},\omega') \\
 \Big\downarrow & & \Big\downarrow \\
 X & = & X
 \end{array}
 $$
of the corresponding isotropic flag varieties bundles.
Going to global vertical fields, we have a (Lie algebra) bundle isomorphism
 $$
 \begin{array}{ccc}
 ad_{\Sp}E & {\too} & ad_{\Sp}{E'} \\
 \Big\downarrow & & \Big\downarrow \\
 X & = & X
 \end{array}
 $$
Using Lemma \ref{endiso} below, it follows that $E'\cong E\otimes M$, for some line
bundle $M$ with $M^2\cong \SO_X$.

As this holds for a generic $E$, it holds for all $E$.
\end{proof}

\begin{lemma}
\label{endiso}
Let $(E, E\otimes E\to L)$ and
$(E',E'\otimes E'\to L')$ be two symplectic vector bundles such
that $ad_{\Sp}\, E$ and $ad_{\Sp}\, E'$ are isomorphic
as Lie algebra bundles. Then there is a line
bundle $M$
such that $E'\cong E\otimes M$.

Furthermore, if we assume $L\cong L'$, then $M^{\otimes 2}\cong \SO_X$.
\end{lemma}

\begin{proof}
Giving a vector bundle $ad_{\Sp}\, E$ with
its Lie algebra structure is equivalent to giving
a principal ${\rm Aut}(\mathfrak{sp}_{2n})$-bundle $P_{{\rm Aut}(\mathfrak{sp}_{2n})}$ which admits
a reduction to a principal ${\rm Gp}_{2n}$-bundle $P_{{\rm Gp}_{2n}}$,
corresponding to $(E,E\otimes E\to L)$.

Since $\mathfrak{sp}_{2n}$ does not have outer automorphisms,
all automorphisms are inner, and ${\rm Aut}(\mathfrak{sp}_{2n})$
is connected. Therefore, we have
$$
{\rm Gp}_{2n} \surj {\rm PGp}_{2n}={\rm Inn}(\mathfrak{sl}_{2n})={\rm Aut}(\mathfrak{sl}_{2n})
$$
Consider the short exact sequence of groups
$$
e \too \CC^* \too {\rm Gp}_{2n} \too {\rm PGp}_{2n} \too e
$$
Hence, the set of reductions of a ${\rm PGp}_{2n}$-bundle
to ${\rm Gp}_{2n}$ is a torsor for the group $H^1(X,\SO_X^*)$.
Therefore, if $(E,E\otimes E\to L)$ is a symplectic
bundle corresponding
to a reduction, the other reductions are
of the form
$$
\big(E\otimes M, (E\otimes M)\otimes (E\otimes M)\too L\otimes M^{\otimes 2}\big)
$$
for any line bundle $M$.

Finally, it follows from this expression that,
if $L\cong L'$, then $M^{\otimes 2}\cong \SO_X$.
\end{proof}

Let $\text{Aut}(X)$ and $\text{Aut}(M_{\Sp}(L))$ be the
automorphisms of $X$ and $M_{\Sp}(L)$ respectively.
Let $J(X)_2$ be the group of line bundles on $X$ of order two.

\begin{proposition}\label{prop1}
There is a natural short exact sequence of groups
$$
e\too J(X)_2 \too {\rm Aut}(M_{\rm Sp}(L)) \too {\rm Aut}(X)
\too e\, .
$$
\end{proposition}

\begin{proof}
{}From the proof of Theorem \ref{thm:autom}, if follows that we have
a surjective homomorphism
$$
\rho\, :\, \text{Aut}(M_{\rm Sp}(L)) \too \text{Aut}(X)\, .
$$
The homomorphism sends $\Phi$ to $\sigma$ (see the
proof of Theorem \ref{thm:autom}).
The kernel of $\rho$ is a quotient of $J(X)_2$. Therefore,
to prove the proposition it suffices to show that the action
of $J(X)_2$ on $M_{\rm Sp}(L)$ is effective.

Let $M_{\rm Sp}(2,L)$ (respectively, $M_{\rm Sp}(2n-2,L)$) be the
moduli space of symplectic bundles of rank $2$ (respectively,
$2n-2$) such that the symplectic form takes values in $L$.
There is a natural embedding 
$$
M_{\rm Sp}(2,L)\times M_{\rm Sp}(2n-2,L)\, \longrightarrow\,
\overline{M}_{\rm Sp}(L)
$$
(where $\overline{M}_{\rm Sp}(L)$ is the moduli space of semistable symplectic bundles),
defined by $((E_1,\varphi_1)\, , (E_2,\varphi_2))\longmapsto
(E_1\oplus E_2,\varphi_1\oplus \varphi_2)$. To prove that the
action
of $J(X)_2$ on $M_{\rm Sp}(L)$ is effective it is enough to show
that the action of $J(X)_2$ on $M_{\rm Sp}(2,L)$ is effective.

First assume that $\deg L\,=\, 2\delta$, where
$\delta$ is an integer. Then for a
general line bundle $M\, \in\, J^\delta(X)$, the symplectic
bundle $M\oplus (L\otimes M^*)\, \in\, M_{\rm Sp}(2,L)$
is moved by the action of every nontrivial element of
$J(X)_2$. Therefore, the action of $J(X)_2$ on $M_{\rm Sp}(2,L)$ is
effective.

Now assume that $\deg L\,=\, 2\delta+1$.
Fix a nontrivial line bundle $\xi\in J(X)_2$. Take
a pair $(E,\theta)$, where $E$ is a stable vector bundle of
rank two with $\bigwedge^2 E = L$, and
$$
\theta \, :\, E\, \longrightarrow\, E\otimes\xi
$$
is an isomorphism. Therefore, $E$ is a fixed point for the action
of $\xi$ on $M_{\rm Sp}(2,L)$.

The line bundle $\xi$ defines a nontrivial \'etale covering
$$
f\,:\, Y\, \longrightarrow\, X
$$
of degree two, and $E$ produces a line bundle $\eta\,\longrightarrow\,
Y$ such that $f_*\eta = E$ (see \cite{BNR}, \cite{Hi}). Therefore,
$\eta$ lies in the Prym subvariety of $J^{2\delta+1}(Y)$ associated
to the covering $f$. The dimension of the Prym variety is
$g-1$. On the other hand, the dimension of $M_{\rm Sp}(2,L)$
is $3g-3$. Since $3g-3 > g-1$, we conclude that the action of
$\xi$ on $M_{\rm Sp}(2,L)$ is effective. This completes the proof
of the proposition.
\end{proof}

\begin{lemma}\label{lem-s-a}
Let $\overline{M}_{\Sp}(L)$ be the moduli space of semistable
symplectic bundles. The automorphism group of $\overline{M}_{\Sp}(L)$ is
identified with ${\rm Aut}(M_{\rm Sp}(L))$.
\end{lemma}

\begin{proof}
The automorphisms of $M_{\rm Sp}(L)$ given by $J(X)_2$ clearly
extend to automorphisms of $\overline{M}_{\Sp}(L)$.
More generally, for any automorphism $\sigma:X\too X$, and
any line bundle $M$ such that $M^{\otimes 2} \cong L\otimes
(\sigma^*L)^\vee$, the automorphism of $M_{\rm Sp}(L)$
defined by
$$
(E, \varphi) \longmapsto (M\otimes \sigma^*E, \beta
\otimes\sigma^*\varphi)\, ,
$$
where $\beta$ is an isomorphism of $M^{\otimes 2}$ with
$\SO_X$, extends to an automorphism of $\overline{M}_{\Sp}(L)$.

On the other hand, from the proof of Theorem \ref{thm:autom}
it follows that any automorphism of the smooth locus of
$M_{\rm Sp}(L)$ extends to an automorphism of
$M_{\rm Sp}(L)$. But the smooth locus of $M_{\rm Sp}(L)$
coincides with the smooth locus of $\overline{M}_{\Sp}(L)$.
Any automorphism of a variety preserves the smooth
locus. Therefore, ${\rm Aut}(\overline{M}_{\Sp}(L))$
is identified with ${\rm Aut}(M_{\rm Sp}(L))$.
\end{proof}

\medskip
\noindent
\textbf{Acknowledgements.}\, This research
was supported by the grant MTM2007-63582 of
the Spanish Ministerio de Ciencia e Innovaci\'on. The second
and third author
thank the hospitality of Tata Institute of Fundamental Research
during the visit where part of this work was done. The first author
thanks McGill University for hospitality while a part of the work
was carried out.


\begin{thebibliography}{EMG}

\bibitem[BNR]{BNR}{A. Beauville, M. S. Narasimhan and S. Ramanan,
}
\textit{Spectral curves and the generalised theta divisor, }
Jour. Reine Angew. Math. \textbf{398} (1989), 169--179.

\bibitem[Bh]{Bh}{U. N. Bhosle, }
\textit{Generalized parabolic bundles and applications--II, }
Proc. Indian Acad. Sci. (Math. Sci.) \textbf{106} (1996), 403--420.

\bibitem[BG]{BG}{I. Biswas and T. L. G\'omez, }
\textit{Hecke correspondence for symplectic bundles with
application to the Picard Bundles, }
Inter. Jour. Math. \textbf{17} (2006), 45--63.

\bibitem[BH]{BH}{I. Biswas and N. Hoffmann, } \textit{A Torelli theorem
for moduli spaces of principal bundles over a curve, }
arXiv:1003.4061.

\bibitem[DKK]{DKK}
J. Draisma, H. Kraft and J. Kuttler,
\textit{Nilpotent subspaces of maximal dimension in semi-simple Lie
algebras, }
Compos. Math. \textbf{142} (2006), 464--476.

\bibitem[Fa]{Fa}{G. Faltings, }
\textit{Stable $G$-bundles and projective connections,}
Jour. Algebraic Geom. \textbf{2} (1993) 507--568.

\bibitem[Hi]{Hi}{N. J. Hitchin, }
\textit{Stable bundles and integrable systems,}
Duke Math. Jour. \textbf{54} (1987), 91--114.

\bibitem[HR]{HR}{J.-M. Hwang and S. Ramanan, }
\textit{Hecke curves and Hitchin discriminant, }
Ann. Sci. \'Ecole Norm. Sup. \textbf{37} (2004), 801--817.

\bibitem[KP]{KP}{A. Kouvidakis and T. Pantev, } \textit{The
automorphism group of the moduli space of semistable
vector bundles, } Math. Ann. 302 (1995), 225--268.

\bibitem[La]{La}{G. Laumon, }
\textit{Un analogue global du c\^one nilpotent, }
Duke Math. J. \textbf{57} (1988), 647--671.

\bibitem[Si]{Si}{C. T. Simpson, }
\textit{Moduli of representations of the fundamental
group of a smooth projective variety, II, }
Publ. Math. I.H.E.S. \textbf{80} (1995), 5--79.

\end{thebibliography}
\end{document}